\newif\ifarxiv
\arxivtrue 

\ifarxiv
\documentclass[USenglish,a4paper]{lipics-v2019-arxiv} 
\nolinenumbers
\else
\documentclass[USenglish,a4paper
]{socg-lipics-v2019} 
\fi

\usepackage{amsmath,amssymb,amsthm}
\usepackage{graphicx}
\usepackage{enumerate}
\usepackage{url,hyperref}

\newcommand\ray{\vec r}
\newcommand\CC{\mathrm{CC}}
\newcommand\Cone{\mathrm{Cone}}
\newcommand\NI{^{\mathrm{D}}} 
\newcommand\chosensides{Q_2'}

\newcommand\observationqed{\qed}

\theoremstyle{plain}\newtheorem{conjecture}
{Conjecture}

\graphicspath{{figures/}} 

\EventLongTitle{This paper is to appear 
  in the proceedings of the
\emph{36th International Symposium on Computational Geometry} (SoCG 2020)
 in June 2016;
Editors: {Sergio Cabello and Danny Chen},
Leibniz International Proceedings in Informatics
}

\EventShortTitle{SoCG'20}
\ArticleNo{XX}

\newtheorem{observation}[theorem]{Observation}

\title{An Almost Optimal Bound on the Number of Intersections of
  Two Simple Polygons}
\titlerunning{The Number of Intersections Between Two Simple Polygons}

\author{Eyal Ackerman}
{Department of Mathematics, Physics, and Computer Science, 
University of Haifa at Oranim, Tivon 36006, Israel}
{ackerman@sci.haifa.ac.il}
{}
{The main part of this work was performed during a visit to Freie Universit\"at Berlin which was supported by the Freie Universit\"at Alumni Program.}

\author{Bal\'azs Keszegh}
{Alfr\'ed R\'enyi Institute of Mathematics,
H-1053 Budapest, Hungary \and MTA-ELTE Lend\"ulet Combinatorial Geometry Research Group}
{keszegh@renyi.hu}
{https://orcid.org/0000-0002-3839-5103}
{Research supported by the Lend\"ulet program of the Hungarian Academy of Sciences (MTA), under the grant LP2017-19/2017 and by the National Research, Development and Innovation Office -- NKFIH under the grant K 116769.}

\author{G\"unter Rote}
{Department of Computer Science,
Freie Universit\"at Berlin, Takustr.\ 9, 14195 Berlin, Germany}
{rote@inf.fu-berlin.de}
{https://orcid.org/0000-0002-0351-5945}
{}

\authorrunning{E. Ackerman, B. Keszegh, and G. Rote}

\keywords{Simple polygon,
Ramsey theory,
combinatorial geometry}
\Copyright{Eyal Ackerman, Bal\'azs Keszegh, and G\"unter Rote}

\ccsdesc[300]{Theory of computation~Computational geometry} 
\ccsdesc[100]{Mathematics of computing~Combinatorial problems}

\begin{document} 


\maketitle
\begin{abstract}

What is the maximum number of intersections of the boundaries
of a simple $m$-gon and a simple $n$-gon,
%
assuming general position?
This is a basic question in combinatorial geometry, and the answer is
easy if at least one of $m$ and $n$ is even:
If both $m$ and $n$ are even, then every pair of sides may cross and
so the answer is $mn$.
If exactly one polygon, say the $n$-gon, 
has an odd number 
of sides, it can intersect each side of the
$m$-gon 
at most $n-1$ times%
; hence there are at most $mn-m$ intersections. It
is not hard to construct examples that meet these bounds.
If both $m$ and $n$ are odd,
the best known construction has $mn-(m+n)+3$ intersections, and it is
conjectured that this is the maximum.  However, the best known upper
bound is only $mn-(m + \lceil \frac{n}{6} \rceil)$, for $m \ge n$.  We
prove a new upper bound of $mn-(m+n)+C
$ for some constant $C$, which is optimal apart from the value of~$C$.
\end{abstract}

\section{Introduction}

To determine the union of two or more geometric
objects in the plane is one of the basic computational geometric problems. In strong relation to that, determining the maximum complexity of the union of two or more geometric objects is a basic extremal geometric problem. We study this problem when the two objects are simple polygons.

Let $P$ and $Q$ be two simple polygons with $m$ and $n$ sides,
respectively,
where $m, n\ge 3$. 
For simplicity we always assume general
position in the sense that no three vertices (of $P$ and $Q$ combined)
lie on a line and no two sides
 (of $P$ and $Q$ combined) are parallel.
We are interested in the maximum number of intersections of the
boundaries of $P$ and $Q$.


This naturally gives an upper bound for the
complexity of the union of the polygon areas as well. (In the worst case all the $m+n$ vertices of the two polygons contribute to the complexity of the boundary in addition to the intersection points.)

This problem was first studied in 1993 by
Dillencourt, Mount, and Saalfeld~\cite{DMS93}. The cases when $m$ or
$n$ is even are solved there. If $m$ and $n$ are both even, then every
pair of sides may cross and so the answer is $mn$.
 Figure \ref{fig:odd-odd-example}a shows one of many ways to achieve
 this number.
If one polygon, say $Q$,
has an odd number $n$
of sides, no line segment~$s$ can be intersected 
$n$~times by $Q$, because otherwise each side of $Q$ would have to
flip from one side of $s$ to the other side.
Thus, each side of the
$m$-gon $P$ is intersected
at most $n-1$ times, for a total
of at most $mn-m$ intersections.
%
It is easy to see that this bound is tight when $P$ has an
even number of sides, see
 Figure \ref{fig:odd-odd-example}b.

When both $m$ and $n$ are odd, the situation is more difficult; the bound that is obtained by the
above argument remains at
$mn-\max\{m,n\}$, 
because the set of $m$ intersections that are necessarily ``missing'' due to
the odd parity of $n$ might conceivably overlap with
the $n$ intersections that are ``missing'' due to
the odd parity of $m$.
However, the best known family of examples gives only 
$mn-(m+n)+3
=(m-1)(n-1)+2$ intersection points, see Figure~\ref{fig:odd-odd-example}c.
Note that in Figure \ref{fig:odd-odd-example},
all vertices of the polygons contribute to the boundary of the
union of the polygon areas.

\begin{conjecture}\label{conj:main}
Let $P$ and $Q$ be simple polygons with $m$ and $n$ sides, respectively, such that $m,n \ge 3$ are odd numbers.
Then there are at most $mn-(m+n)+3$ intersection points between sides of $P$ and sides of $Q$.
\end{conjecture}

In \cite{DMS93} an unrecoverable error appears in a claimed proof of Conjecture~\ref{conj:main}.
Another attempted proof~\cite{FG12} also turned out to have a fault.
The
only correct improvement over the trivial upper bound is an upper bound of $mn-(m + \lceil
\frac{n}{6} \rceil)$ for $m \ge n$, due to
\v{C}ern\'y, K\'ara, Kr\'al', Podbrdsk\'y, Sot\'akov\'a,
and \v{S}\'amal~\cite{czechs}.
We will briefly discuss their proof in Section~\ref{overview}.

We improve the upper bound to $mn-(m+n)+O(1)$, which is optimal apart from an additional constant:

\begin{theorem}\label{thm:main}
There is an absolute constant $C$ such that the following holds.
Suppose that $P$ and $Q$ are simple polygons with $m$ and $n$ sides, respectively,
such that $m$ and $n$ are odd numbers.
Then there are at least $m+n-C$ pairs of a side of $P$ and a side of
$Q$ that do not intersect.
Hence, there are at most $mn-(m+n)+C$ intersections.
\end{theorem}

The value of the constant $C$ that we obtain in our proof is around $2^{2^{67}}\!$.
We did not 
make a large effort
to optimize this value, 
and
 obviously, 
there is ample space for improvement.

\begin{figure}
	\centering
	\includegraphics[scale=1]{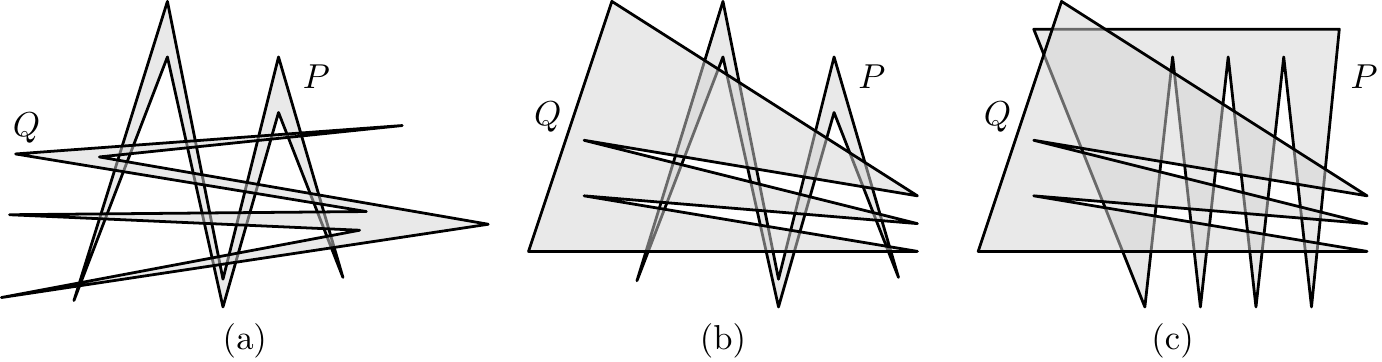}
	\caption{%
          (a)~Optimal construction for $m=n= 8$, with $8\times 8=64$ intersections.
          (b)~Optimal construction for $m= 8$, $n=7$, with $8\times 6=48$ intersections.
          (c) Lower-bound construction for $m=9$, $n=7$. There are
          $8\times 6 + 2 = 50$ intersections.
        }
	\label{fig:odd-odd-example}
\end{figure}

\section{Overview of the Proof}
\label{overview}

First we establish the crucial statement that the odd parity of $m$ and $n$
allows us to \emph{associate} to any two consecutive sides of one polygon
a pair of
consecutive sides of the other polygon
with a restricted intersection pattern among the four involved sides
(Lemma~\ref{lem:1} and
Figure~\ref{fig:hooked}). This is the only place where we use the odd
parity of the polygons.

A simple observation
(Observation~\ref{obs:CCs})
relates the bound on $C$ in Theorem~\ref{thm:main}
to the number of connected components of the bipartite ``disjointness graph''
between the polygon sides of $P$ and $Q$.
Our goal is therefore to show that there are few connected components.

We proceed to consider \emph{two}
pairs of associated pairs of sides (4 consecutive pairs with 8 sides in total).
Unless they form a special structure, they cannot belong to four
different connected components
(Lemma~\ref{lem:2}).
(Four is the maximum number of components that they could conceivably have.)
The proof involves a case distinction with a moderate amount of cases.
This structural statement allows us to reduce the bound on the number
of components by a constant factor, and thereby, we can already improve the
best previous result on the number of intersections
(Proposition~\ref{weak} in Section~\ref{sec:weaker}).

Finally, to get a constant bound on the number of
components,
our strategy is to use Ramsey-theoretic arguments
like the
Erd\H{o}s--Szekeres Theorem on caps and cups or the
pigeonhole principle
(see Section~\ref{sec:Ramsey})
in order to
 impose
additional structure on the configurations that we have to analyze.
This is the place in the argument where we give up control over the
constant~$C$ in exchange for useful properties that allow us to derive
a contradiction. 
This eventually boils down again to a moderate number of cases
(Section~\ref{sec:final-cases}).

By contrast, the proof
of the bound
 $mn-(m + \lceil
 \frac{n}{6} \rceil)$ for $m \ge n$
by
\v{C}ern\'y et al.\ 
 proceeds in a more local manner.
 The core of their argument
\cite[Lemma~3]{czechs}, which is proved by case distinction,
 is that it is impossible to have
 6 consecutive sides of one polygon together with
 6 distinct sides of the other polygon forming a perfect matching in the disjointness graph.
This statement is used to bound the number of components of the
  disjointness graph. (Lemma~\ref{cor:nCC} below uses a similar argument.)

\section{An Auxiliary Lemma on Closed Odd Walks}
\label{subsec:odd-walks}

We begin with the following seemingly unrelated claim concerning a specific small edge-labeled multigraph.
Let $G_0=(V_0,E_0)$ be the undirected multigraph
shown in Figure~\ref{fig:Gprime}.
It has four nodes $V_0=\{{\rm I},{\rm II},{\rm III},{\rm IV}\}$
and five edges $E_0=\{e_1=\{{\rm II},{\rm IV}\}, e_2=\{{\rm I},{\rm IV}\}, e_3=\{{\rm I},{\rm II}\}, e_4=\{{\rm I},{\rm III}\}, e_5=\{{\rm I},{\rm III}\}\}$.
 Every edge $e_i \in E_0$ has a label $L(e_i) \in \{a,b,*\}$ as follows:
$L(e_1)=*$, $L(e_2)=L(e_4)=a$, $L(e_3)=L(e_5)=b$.
\begin{figure}
  \begin{minipage}[t]{5.2cm}
    \centering
    \includegraphics[height=3cm]{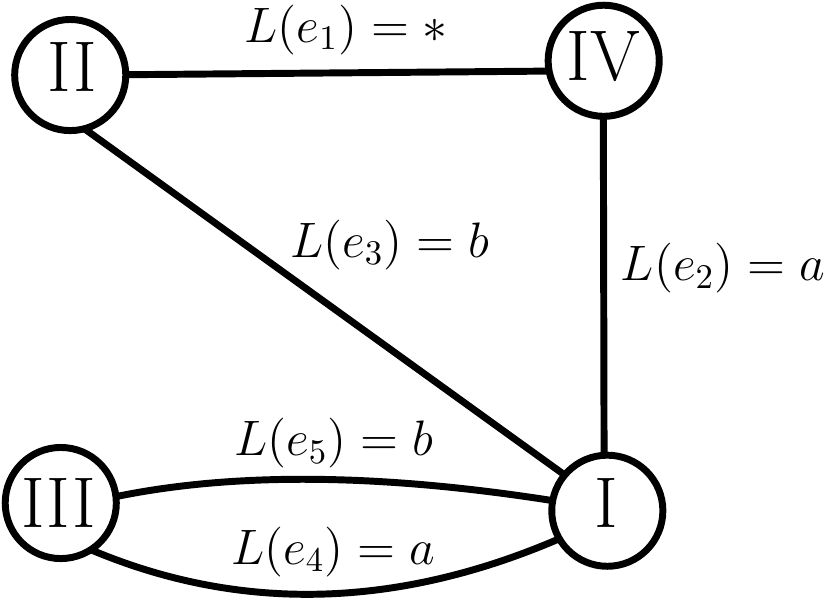}
    \caption{The edge-labeled multigraph $G_0$ in
     Proposition~\ref{prop:domino}.}
    \label{fig:Gprime}
  \end{minipage}
  \qquad
  \begin{minipage}[t]{8cm}
    \centering
    \includegraphics[height=3cm]{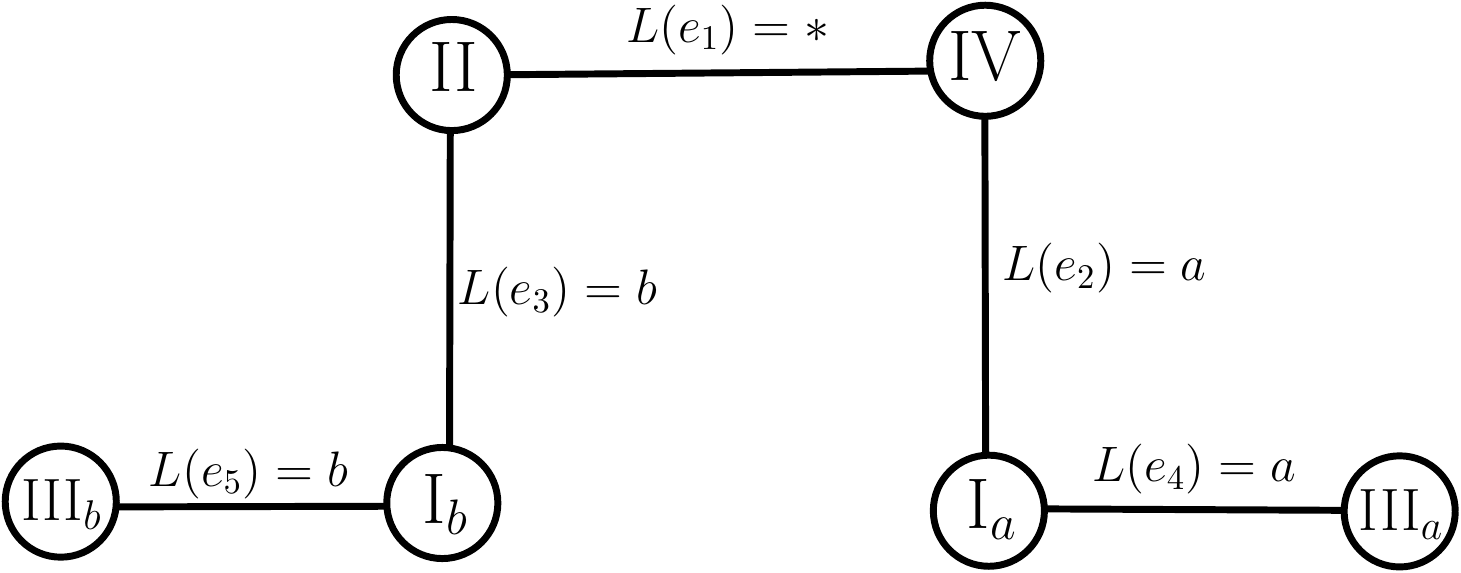}
    \caption{The unfolded graph $G_0'$}\label{fig:Gprimeprime}
  \end{minipage}
\end{figure}

\begin{proposition}\label{prop:domino}
If $W$ is a closed walk in $G_0$ of odd length,
then $W$ contains two cyclically consecutive edges of labels $a$ and $b$.
\end{proposition}

\begin{proof}
Suppose for contradiction that $W$ does not contain two consecutive
edges of labels $a$ and $b$.
Since $W$ cannot switch between the $a$-edges and the $b$-edges in
I or III,
we can split $\rm I$
(resp., $\rm III$) into two nodes ${\rm I}_a$ and ${\rm III}_b$
(resp., ${\rm III}_a$ and ${\rm III}_b$) such that every $a$-labeled
edge that is incident to ${\rm I}$  (resp., ${\rm III}$) in $G_0$ becomes
incident to ${\rm I}_a$ (resp., ${\rm III}_a$) and every
$b$-labeled edge that is incident to ${\rm I}$  (resp., ${\rm III}$)
in $G_0$ becomes incident to ${\rm I}_b$ (resp., ${\rm III}_b$).
In the resulting graph $G_0'$, which is shown in
Figure~\ref{fig:Gprimeprime},
we can find a closed walk $W'$ that corresponds to $W$ and that uses the
 edges with the same name as~$W$.
Since $G_0'$ is a path, every closed walk has even length.
Thus, $W$ cannot have odd length.
\end{proof}

\section{General Assumptions and Notations}

Let $P$ and $Q$ be two simple polygons with sides
 $p_0,p_1,\ldots,p_{m-1}$ and $q_0,q_1,\ldots,q_{n-1}$.
 We assume that $m\ge3$ and $n\ge3$ are odd numbers. 
Addition and subtraction of indices is modulo $m$ or $n$,
respectively.
We consider the sides $p_i$ and $q_j$ as closed line segments.
The condition that the polygon $P$ is simple means that
its edges are pairwise disjoint except for the unavoidable
common endpoints between \emph{consecutive} sides $p_i$  and $p_{i+1}$.
Throughout this paper, unless stated otherwise, we regard a polygon
as a piecewise linear closed curve, and we disregard the region that it
encloses. Thus, by intersections between $P$ and $Q$, we mean
intersection points between the polygon \emph{boundaries.}

As mentioned, we assume that the vertices of $P$ and $Q$ are in
general position (no three of them on a line), and so every
intersection point between $P$ and $Q$ is an interior point
of two polygon sides.

\subparagraph{The Disjointness Graph.}
As in \cite{czechs},
our basic tool of analysis is
the \emph{disjointness graph}
of $P$ and $Q$,
which we denote by $G\NI=(V\NI,E\NI)$.
(Its original name in \cite{czechs} is \emph{non-intersection graph}.)
It is a bipartite graph with node set $V\NI=\{p_0,p_1,\ldots,p_{m-1}\} \cup \{q_0,q_1,\ldots,q_{n-1}\}$
and edge set $E\NI=\{\,(p_i,q_j) \mid p_i \cap q_j = \emptyset\,\}$.
(Since we are interested in the situation where almost all pairs of edges
intersect,
the {disjointness graph} is more useful than its more commonly used complement, the
intersection graph.)
\looseness-1 

Our goal is to bound from above the number of connected components of $G\NI$.

\begin{observation}\label{obs:CCs}
If $G\NI$ has at most $C$ connected components, then $G\NI$ has at least $m+n-C$ edges.
Thus, there are at least $m+n-C$ pairs of a side of $P$ and a side of
$Q$ that do not intersect, and there are at most
$mn-(m+n)+C$ crossings between $P$ and $Q$.
\observationqed
\end{observation}

\subparagraph{Geometric Notions.}

Let $s$ and $s'$ be two line segments.
We denote by $\ell(s)$ the line through $s$ and by $I(s,s')$ the
intersection of $\ell(s)$ and $\ell(s')$
see Figure~\ref{fig:lem1}.
We say that $s$ and $s'$ are \emph{avoiding} if neither of them
contains $I(s,s')$.
(This requirement is stronger than just disjointness.)
If $s$ and $s'$ are avoiding
or share an endpoint, we denote
by $\ray_{s'}(s)$ the
ray from $I(s,s')$ to infinity that contains~$s$,
and
by $\ray_s(s')$ the
ray from $I(s,s')$ to infinity that contains~$s'$.
Moreover, we denote by $\Cone(s,s')$ the convex cone with apex
$I(s,s')$
between these two rays.

\begin{observation}\label{obs:endpoint-in-cone}
  If a segment $s''$ that does not go through $I(s,s')$
  has one of its endpoints in the interior of $\Cone(s,s')$, then
$s''$ cannot intersect both
$\ray_{s'}(s)$ and
$\ray_{s}(s')$.
In particular, it cannot intersect
both $s$ and $s'$.
\observationqed
\end{observation} 

For a polygon side $s$ 
of $P$ or $Q$,
$\CC(s)$ denotes the connected component of the disjointness graph
$G\NI$ to which $s$
belongs.

\subsection{Associated Pairs of Consecutive Sides}

\begin{lemma}\label{lem:1}
Let $p_a$ and $p_b$ be two sides of $P$ that are either consecutive or avoiding such that $\CC(p_a) \ne \CC(p_b)$. 
Then there are two consecutive sides $q_i,q_{i \pm 1}$ of $Q$ such that $(p_a,q_i), (p_b,q_{i\pm 1}) \in E\NI$ and  $(p_a,q_{i\pm 1}), (p_b,q_{i}) \notin E\NI$.
Furthermore, $I(p_a,p_b) \in \Cone(q_i,q_{i \pm 1})$ or $I(q_i,q_{i \pm 1}) \in \Cone(p_a,p_b)$.
\end{lemma}

The sign `$\pm$' is needed since we do not know which of the consecutive sides intersects $p_i$ and is disjoint from $p_{i+1}$.

\begin{proof}
We may assume without loss of generality that $I(p_a,p_b)$ is the origin, $p_a$ lies on the positive $x$-axis and the interior of $p_b$ is above the $x$-axis.
The lines $\ell(p_a)$ and $\ell(p_b)$ partition the plane into four
convex
cones (``quadrants'').
Denote them in counterclockwise order by ${\rm I},{\rm II},{\rm
  III},{\rm IV}$,
starting with ${\rm I}=\Cone(p_a,p_b)$, see Figure~\ref{fig:lem1}.
\begin{figure}[htb]
    \centering
    \includegraphics{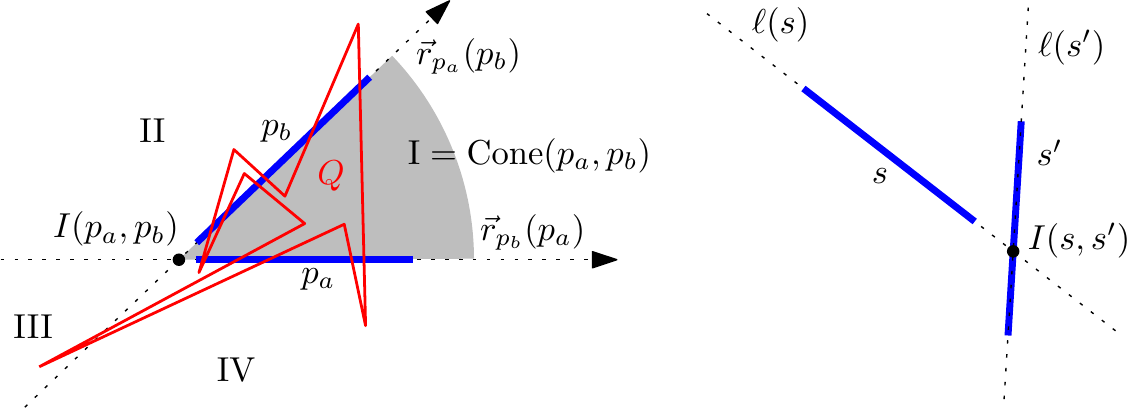}
    \caption{
      How an odd polygon $Q$ can intersect two segments. The segments
    $p_a$ and $p_b$ are avoiding, whereas $s$ and $s'$ are disjoint but non-avoiding.}
	\label{fig:lem1}
\end{figure}
Every side of $Q$ must intersect $p_a$ or $p_b$ (maybe both), since $\CC(p_a) \ne \CC(p_b)$.
One can now check that traversing the sides of $Q$ in order generates a closed
walk $W$ in the graph $G_0$
of Figure~\ref{fig:Gprime}. 
For example, a side of $Q$ that we traverse from its endpoint in ${\rm
  I}$ to its endpoint in ${\rm III}$ and that intersects $p_a$ corresponds
to traversing the edge $e_4=\{{\rm I},{\rm III}\}$
from ${\rm I}$ to ${\rm III}$,
whose label is $L(e_4)=a$.
We do not care which of $p_a$ and $p_b$ are crossed
when we move between
${\rm II}$ and~${\rm IV}$.

It follows from Proposition~\ref{prop:domino} that $Q$ has
two consecutive sides $q_i,q_{i \pm 1}$ such that $q_i$ intersects
$p_b$ and does not intersect $p_a$, while $q_{i \pm 1}$ intersects $p_a$ and does not intersect~$p_b$. Hence, $(p_a,q_i), (p_b,q_{i\pm 1}) \in E\NI$ and  $(p_a,q_{i\pm 1}), (p_b,q_{i}) \notin E\NI$.
Furthermore, $I(q_i,q_{i \pm 1})$ must be either in ${\rm I}$ or ${\rm III}$ as these are the only nodes in $G_0$ that are
incident both to an edge labeled $a$ and an edge labeled $b$.
In the latter case $I(p_a,p_b) \in \Cone(q_i,q_{i \pm 1})$, and in the former case $I(q_i,q_{i \pm 1}) \in \Cone(p_a,p_b)$.
\end{proof}

Let $p_i,p_{i+1}$ be two sides of $P$ such that $\CC(p_i) \ne \CC(p_{i+1})$.
Then by Lemma~\ref{lem:1} there are sides $q_j,q_{j \pm 1}$  of $Q$
such that $(p_i,q_j), (p_{i+1},q_{j \pm 1}) \in E\NI$.
We say that the pair $q_j,q_{j \pm 1}$ is \emph{associated} to $p_i,p_{i+1}$.
By Lemma~\ref{lem:1} we have $I(q_j,q_{j \pm 1}) \in \Cone(p_i,p_{i+1})$ or $I(p_i,p_{i+1}) \in \Cone(q_j,q_{j \pm 1})$.
If the first condition holds we say that $p_i,p_{i+1}$ is
\emph{hooking} and $q_j,q_{j \pm 1}$ is \emph{hooked},
see Figure~\ref{fig:hooked}.
 In the
second case we say that $p_i,p_{i+1}$ is \emph{hooked} and $q_j,q_{j \pm
  1}$ is \emph{hooking}. Note that it is possible that a pair of
consecutive sides is both hooking and hooked (with respect to two
different pairs from the other polygon or even with respect to a
single pair, as in Figure~\ref{fig:hooked}c).

\begin{figure}[htb]
  \centering
  \includegraphics{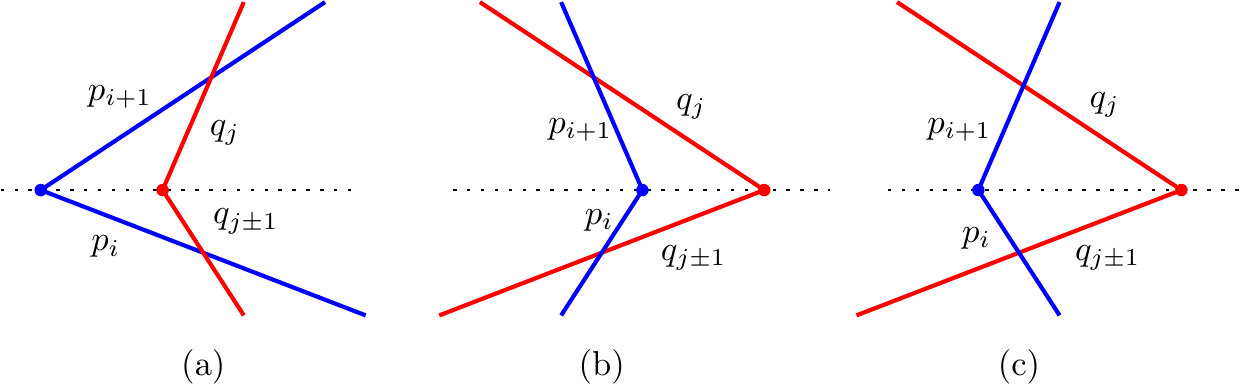}
  \caption{Hooking and hooked pairs of consecutive sides.
(a)~The pair $p_i,p_{i+1}$ is
 {hooking} and the associated pair $q_j,q_{j \pm 1}$ is {hooked}.
 (b)~vice versa.
 (c)~Both pairs are both hooking and hooked.
  }
  \label{fig:hooked}
\end{figure}

\begin{observation}[The Axis Property]
  \label{lem:axis}
  If the pair $p_i,p_{i+1}$ and the pair
  $q_j,q_{j \pm 1}$ are {associated} such that $(p_i,q_j), (p_{i+1},q_{j \pm 1}) \in E\NI$, then the line
  through
  $I(p_i,p_{i+1})$ and  $I(q_j,q_{j \pm 1})$
separates
$p_i$ and  $q_{j \pm 1}$ on the one side from
$p_{i+1}$ and  $q_j$ on the other side.
\observationqed
\end{observation}
We call this line the \emph{axis} of the associated pairs. In our
figures
it appears as a dotted line when it is shown.

\section{The Principal Structure Lemma about Pairs of Associated Pairs}

\begin{lemma}\label{lem:2}
	Let $p_i,p_{i+1},p_j,p_{j+1}$ be two pairs of consecutive sides of $P$
	that belong to four different connected components of $G\NI$.
	Then it is impossible that both $p_i,p_{i+1}$ and $p_j,p_{j+1}$ are hooked or that both pairs are hooking.
\end{lemma}

\begin{proof}
Suppose first that both pairs $p_i,p_{i+1}$ and $p_j,p_{j+1}$, are hooking
and let $q_{i'},q_{i'\pm 1}$ and $q_{j'},q_{j'\pm 1}$ be their associated (hooked) pairs such that: 
$(p_i, q_{i'}), (p_{i+1},q_{i'\pm 1}) \in E\NI$, 
$(p_j, q_{j'}), (p_{j+1},q_{j'\pm 1}) \in E\NI$,
$I(q_{i'},q_{i'\pm 1}) \in \Cone(p_i,p_{i+1})$ and $I(q_{j'},q_{j'\pm 1}) \in \Cone(p_j,p_{j+1})$.

For better readability, we rename
$p_i,p_{i+1}$ and $q_{i'},q_{i'\pm 1}$ as $a,b$ and $A,B$, and we rename
$p_j,p_{j+1}$ and $q_{j'},q_{j'\pm 1}$ as $a',b'$ and $A',B'$.
The small letters denote sides of $P$ and the capital letters denote
sides of $Q$.
In the new notation,
$a,b$ are consecutive sides of $P$ with an associated pair $A,B$ of
consecutive sides of $Q$,
and
$a',b'$ are two other consecutive sides of $P$ with an associated pair $A',B'$ of
consecutive sides of $Q$.
The disjointness graph $G\NI$ contains the edges
$(a, A), (b,B),(a', A'), (b',B')$.
Since $a,b,a',b'$ belong to different connected components of  $G\NI$,
it follows that the nodes $A,B,A',B'$, to which they are
connected, 
belong to the same four different connected components.
There can be no more edges among these
 eight nodes, and they induce a
matching in $G\NI$.
One can remember as a rule that
every side of $P$ intersects every side of $Q$ among the eight
involved sides, except when their
names differ only in their capitalization.
In particular, each of $A'$ and $B'$ intersects each of
$a$ and~$b$.
and hence they must lie as in Figure~\ref{fig:normalize}a.
To facilitate the future discussion,
we will now normalize the positions of these four sides.

\begin{figure}
	\centering
	\includegraphics[scale=0.92]{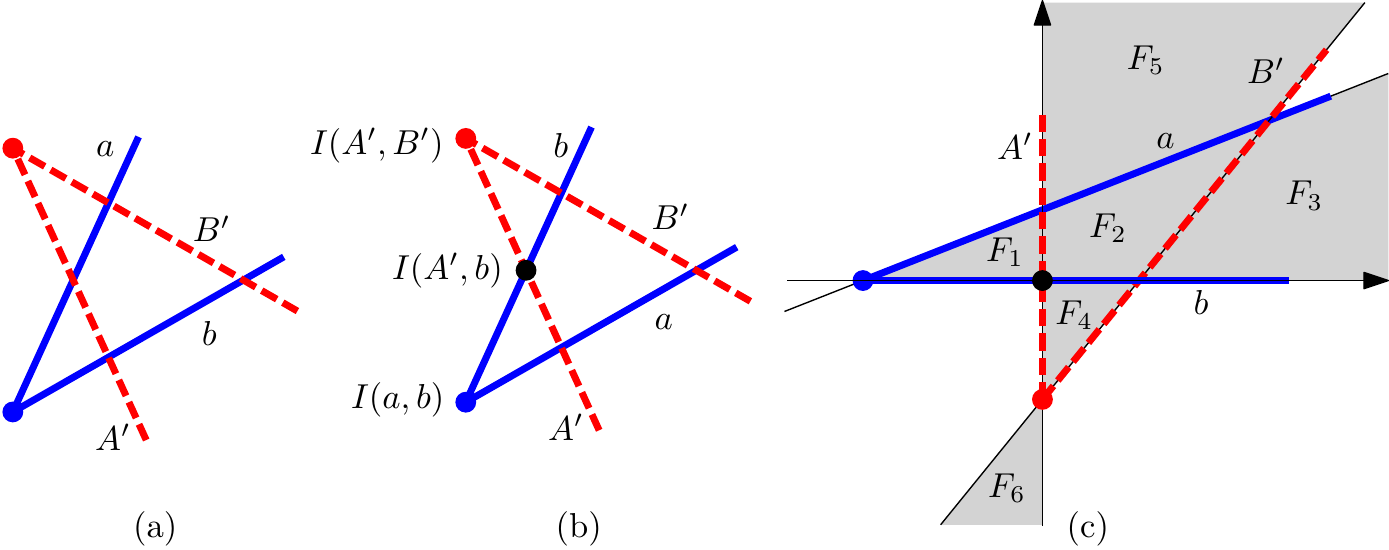}
	\caption{Normalizing the position of $a,b,A',B'$}
	\label{fig:normalize}
\end{figure}


We first ensure that
the intersection $I(A',b)$ is directly adjacent to the
two polygon vertices $I(a,b)$ and $I(A',B')$ in the arrangement
of the four sides, as shown in
Figure~\ref{fig:normalize}b.
This can be achieved 
by swapping the labels $a,A$ with the labels $b,B$ if necessary,
and by
independently swapping the labels $a',A'$ with $b',B'$ if necessary.
Our assumptions are invariant under these swaps.

By an affine transformation
we may finally assume that
$I(A',b)$ is the origin;
$b$ lies on the $x$-axis and is directed to the right;
and $A'$ lies on the $y$-axis and is directed upwards.
Then $a$ has a positive slope and its interior is in the upper half-plane,
and $B'$ has a positive slope and its interior is to
the right of the $y$-axis, see Figure~\ref{fig:normalize}c.

\begin{figure}[htb]
  \begin{minipage}[b]{0.5\textwidth}
    \centering
  \includegraphics[scale=1]{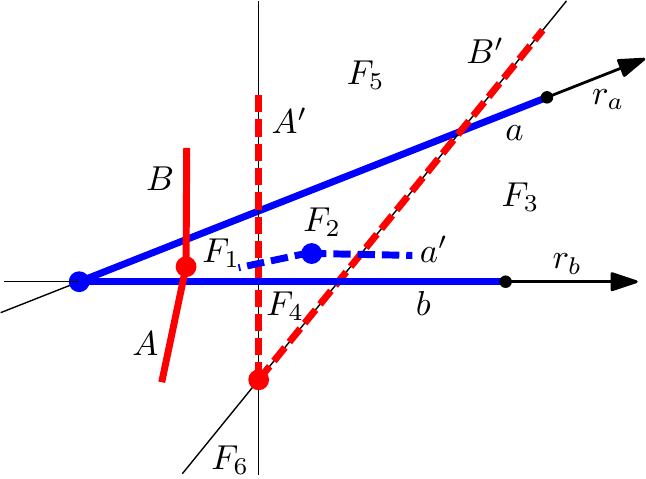}
  \caption{Case~\ref{case-1-2}: $I(A,B) \in F_1$, 
    $I(a',b')
          \in F_2$}
        \label{fig:1_2}
  \end{minipage}
  \begin{minipage}[b]{0.5\textwidth}
    \centering
  \includegraphics[scale=1]{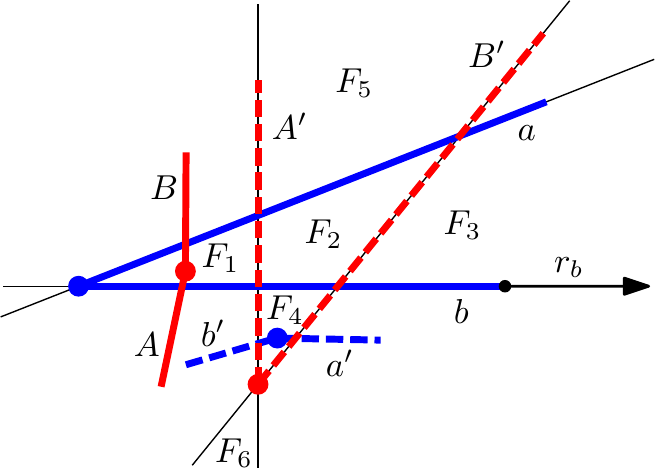}
  \caption{Case~\ref{case-1-4}: $I(A,B) \in F_1$, 
    $I(a',b') \in
  F_4$}
\label{fig:1_4}
  \end{minipage}
\end{figure}

The arrangement of the lines through $a,b,A',B'$
has $11$ faces, some of which are marked as $F_1,\ldots,F_6$ in
Figure~\ref{fig:normalize}.
Our current assumption is that both $a,b$ and $a',b'$ are hooking:
The hooking of $a,b$ means that
$I(A,B) \in \Cone(a,b) =
F_1 \cup F_2 \cup F_3$.
By the Axis Property (Observation~\ref{lem:axis}), the line
through $I(A',B')$ and
$I(a',b')$ must separate $A'$ from $B'$.
Therefore, the vertex
$I(a',b')$ can lie only in $F_2 \cup F_4 \cup F_5 \cup F_6$.
Thus, based on the faces that contain $I(A,B)$ and $I(a',b')$, there are $12$ cases to consider.
Some of these cases are symmetric, and all can be easily
dismissed, as follows.

In the figures, the four sides
$a',b',A',B'$, which are associated to the second associated pair
are dashed. All dashed sides of one polygon must intersect all solid
sides of the other polygon.

\begin{enumerate}

\item
\label{case-1-2}
  $I(A,B) \in F_1$ and $I(a',b') \in F_2$,
see
Figure~\ref{fig:1_2}
  (symmetric to $I(A,B) \in F_2$ and $I(a',b') \in F_4$).
Let $r_{a}$ (resp., $r_b$) be the ray on $\ell(a)$ (resp.,
$\ell(b)$) that goes from the right endpoint
of $a$ (resp., $b$) to the right.
Since $a'$ is not allowed to cross $b$,
  the only way for $a'$ to intersect $A$ is by crossing $r_{b}$.
Similarly, in order to intersect
$B$, $a'$ has to cross $r_{a}$.
However,
it cannot intersect both $r_{a}$ and~$r_{b}$, by
Observation~\ref{obs:endpoint-in-cone}.

Since we did not use the assumption that $A,B$ are hooked, the
analysis holds for the symmetric Case~\ref{case-2-4},
$I(A,B) \in F_2$ and $I(a',b') \in F_4$, as well.
  
\item
\label{case-1-4}
  $I(A,B) \in F_1$ and $I(a',b') \in F_4$,
see
Figure~\ref{fig:1_4}.
Since $a'$ is not allowed to cross $b$,
the only way for $a'$ to intersect $B$ is by crossing $r_{b}$.
However, in this case $a'$ cannot intersect $A$.


\begin{figure}[htb]
	\centering
	\includegraphics[scale=1]{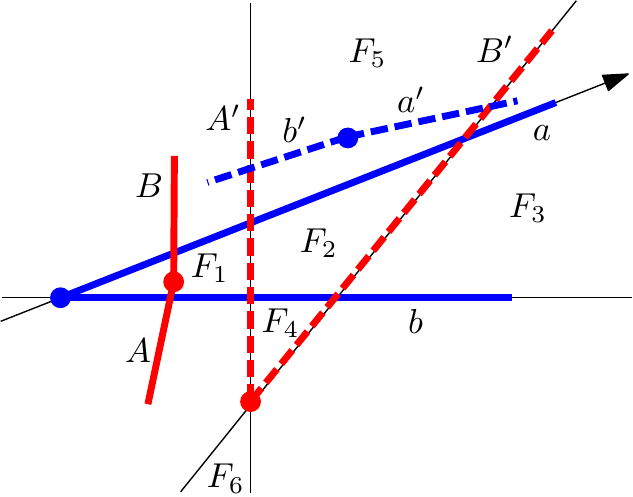}
	\caption{Case~\ref{case-1-5}:
          $I(A,B) \in F_1$ and $I(a',b') \in F_5$}
	\label{fig:1_5}
\end{figure}

\item
\label{case-1-5}
  $I(A,B) \in F_1$ and $I(a',b') \in F_5$, see Figure~\ref{fig:1_5}
  (symmetric to $I(A,B) \in F_3$ and $I(a',b') \in F_4$).
  Both $a'$ and $b'$ must intersect $A$, and they have to go below the
  line $\ell(b)$ to do so.
  However,
 $a'$ can only cross $\ell(b)$ to the right of $b$,
  and  $b'$ can only cross $\ell(b)$ to the left of~$b$,
  and therefore they cross $A$ from different sides.
  This is impossible, because $a'$ and $b'$ start from
  the same point.


\begin{figure}[tb]
  \begin{subfigure}[b]{0.48\textwidth}
    \centering
    \includegraphics{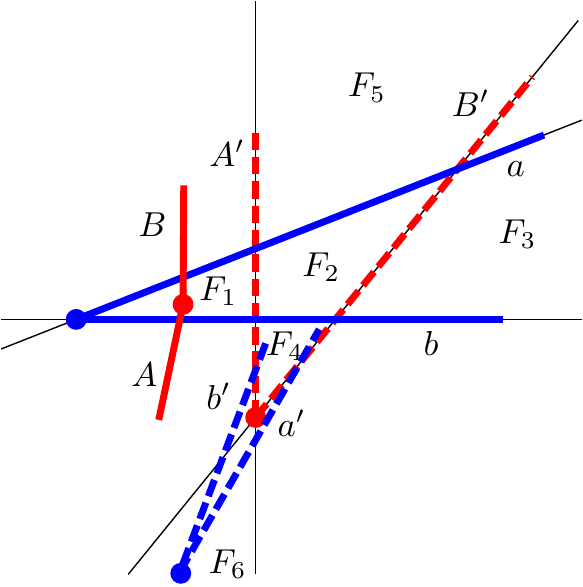}
    \caption{At least one of the sides $a'$ and $b'$ has an endpoint in $F_4$.}\label{fig:1_6a}
  \end{subfigure}
  \qquad
  \begin{subfigure}[b]{0.5\textwidth}
    \centering
    \includegraphics{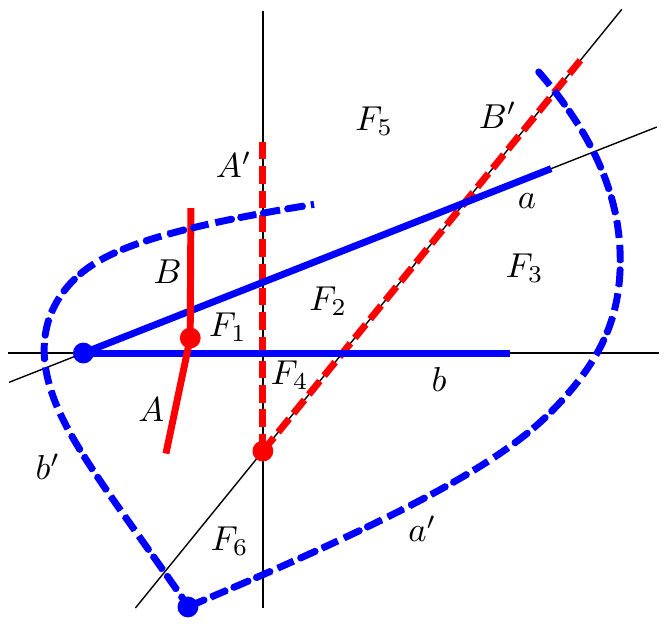}
    \caption{None of the sides $a'$ and $b'$ has an
      endpoint in~$F_4$.}
        \label{fig:1_6b}
  \end{subfigure}
  \caption{Case~\ref{case-1-6}: $I(A,B) \in F_1$
(or $I(A,B) \in F_2$, which is similar)
    and $I(a',b')
          \in F_6$.
        }
	\label{fig:1_6}
\end{figure} 
  
\item
\label{case-1-6}
  $I(A,B) \in F_1$ and $I(a',b') \in F_6$.
If one of the polygon sides $a'$ and $b'$ has an endpoint in
$F_4$ (see Figure~\ref{fig:1_6a}), then this side cannot intersect
$B$.
So assume otherwise, see Figure~\ref{fig:1_6b}.
The side $a'$
intersects $B'$ and is disjoint from $A'$, while
$b'$
is disjoint from $B'$ and intersects $A'$.
(Due to space limitation some line segments are drawn schematically as curves.)
Thus, each of $a'$ and $b'$ has an endpoint in $F_2 \cup F_5$.
But then $I(A,B) \in \Cone(a',b')$ and it follows
from
Observation~\ref{obs:endpoint-in-cone} that neither $A$ nor $B$ can
intersect both $a'$ and $b'$.

\item
\label{case-2-2}
  $I(A,B) \in F_2$ and $I(a',b') \in F_2$,
  see Figure~\ref{fig:2_2}.
Since $a',b'$ is hooking,
$I(A',B') \in \Cone(a',b')$,
and
the line segments
$a',b',A',b,B'$
enclose a convex pentagon.
%
\begin{figure}[htb]
  \begin{minipage}[b]{0.50\textwidth}
  \centering
    \includegraphics{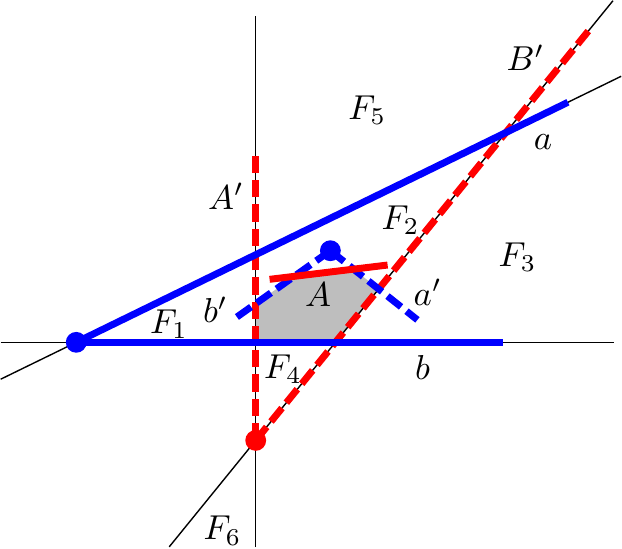}
    \caption{Case~\ref{case-2-2}: $I(A,B) \in F_2$, 
      $I(a',b') \in\nobreak
  F_2$}
\label{fig:2_2}
\end{minipage}
\qquad
  \begin{minipage}[b]{0.5\textwidth}
    \centering\includegraphics{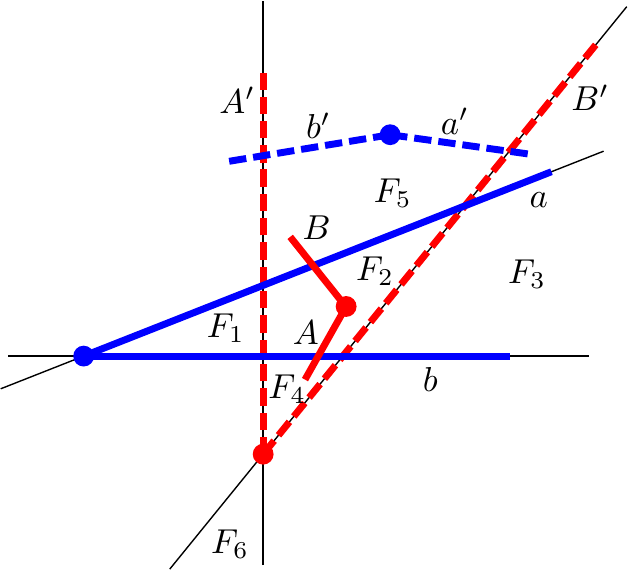}
    \caption{Case~\ref{case-2-5}: $I(A,B) \in F_2$, 
      $I(a',b') \in
  F_5$}
\label{fig:2_5}
\end{minipage}
\end{figure}
The polygon side $A$
must intersect $b$, $a'$ and $b'$, but it
 is restricted to $F_2 \cup F_4$.
It follows that $A$ must intersect three sides of the
pentagon, which is impossible.  (This is in fact the
only place where we need the assumption that 
$a',b'$ is hooking.)

\item
\label{case-2-4}
  $I(A,B) \in F_2$ and $I(a',b') \in F_4$.
  This is symmetric to
Case~\ref{case-1-2}. 

\item
\label{case-2-5}
  $I(A,B) \in F_2$ and $I(a',b') \in F_5$,
see Figure~\ref{fig:2_5}
  (symmetric to $I(A,B) \in F_3$ and $I(a',b') \in F_2$).
Then $A$ is restricted to $ F_2 \cup F_4$, while $a'$ and $b'$ do not intersect $F_2$ and $F_4$.
Therefore $A$ can intersect neither $a'$ nor $b'$. 


\item $I(A,B) \in F_2$ and $I(a',b') \in F_6$.
  This case is very similar to Case
\ref{case-1-6},
  where $I(A,B) \in F_1$ and $I(a',b') \in
  F_6$,
see Figure~\ref{fig:1_6}.
If one of the polygon sides $a'$ and $b'$ has an endpoint in
$F_4$, then it cannot intersect $B$.
Otherwise, $I(A,B) \in \Cone(a',b')$ and
therefore,
neither $A$
nor $B$ can
intersect both $a'$ and $b'$.

\item  $I(A,B) \in F_3$ and $I(a',b') \in F_2$.
  This is symmetric to
    Case~\ref{case-2-5}.
\item  $I(A,B) \in F_3$ and $I(a',b') \in F_4$.
  This is symmetric to
    Case~\ref{case-1-5}.


\begin{figure}[htb]
  \centering
\includegraphics{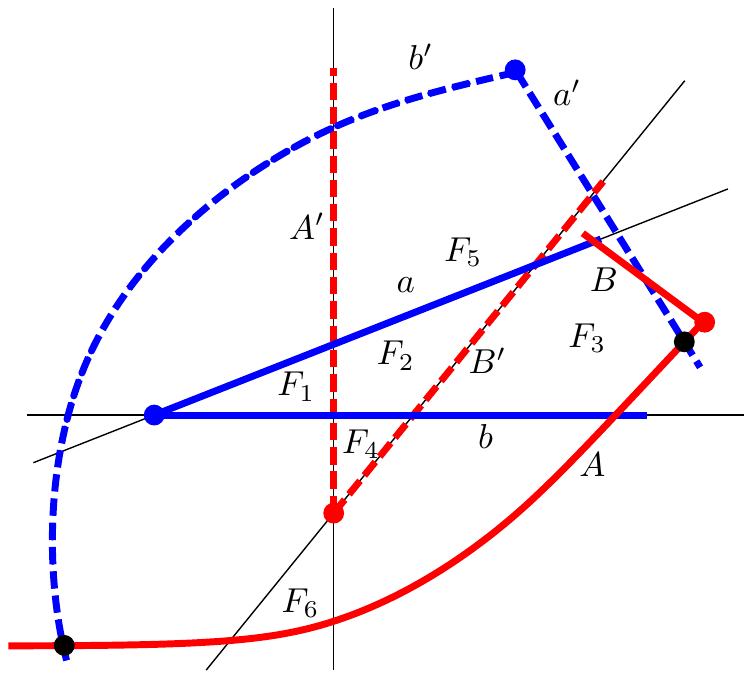}    
  \caption{Case~\ref{case-3-5}: $I(A,B) \in F_3$ and $I(a',b') \in F_5$}
  \label{fig:3_5}
\end{figure}
  \item
\label{case-3-5}
    $I(A,B) \in F_3$ and $I(a',b') \in F_5$,
see Figure~\ref{fig:3_5}.
Then the intersection of $b'$ and~$A$ can lie only in the lower left
quadrant. 
It follows that the triangle whose vertices are $I(a',b')$,
$I(a',A)$ and $I(A,b')$
contains $a$ and does not contain $I(A,B)$.
This in turn implies that $B$ cannot intersect both
$b'$ and $a$,
without intersecting $B'$.


\item
\label{case-3-6}
  $I(A,B) \in F_3$ and $I(a',b') \in F_6$,
  see Figure~\ref{fig:3_6}.
As in Case
\ref{case-1-6},
we may assume that neither $a'$ nor $b'$ has an endpoint in $F_4$,
since then this side could not intersect $B$.
We may also assume that $I(A,B) \notin \Cone(a',b')$ for otherwise neither $A$ nor $B$
intersects both of $a'$ and $b'$, according to
Observation~\ref{obs:endpoint-in-cone}.
If $a'$ has an endpoint in $F_2$, then it cannot intersect $B$ (see Figure~\ref{fig:3_6a}).
Otherwise, if $a'$ has an endpoint in $F_5$, then $B$ cannot intersect $b'$ (see Figure~\ref{fig:3_6b}).
\end{enumerate}

\begin{figure}[htb]
  \begin{subfigure}[b]{0.45\textwidth}
    \centering
    \includegraphics{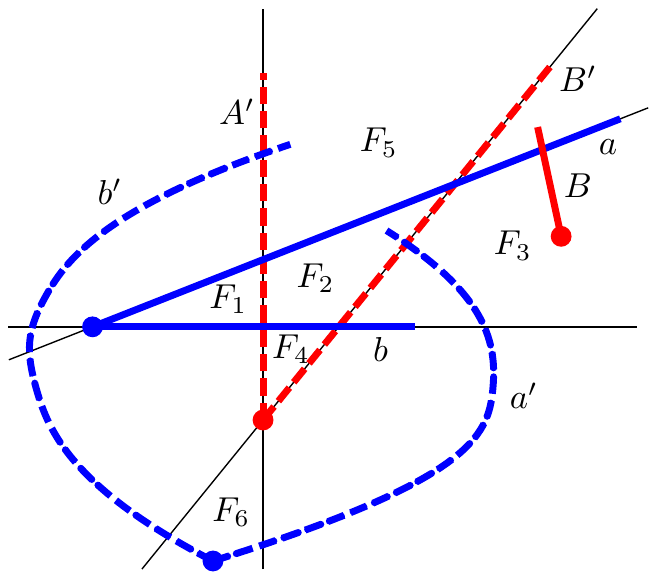}
  \caption{If $a'$ has an endpoint in $F_2$, then it cannot intersect $B$.}\label{fig:3_6a}
\end{subfigure}
\qquad
  \begin{subfigure}[b]{0.5\textwidth}
    \centering
    \includegraphics{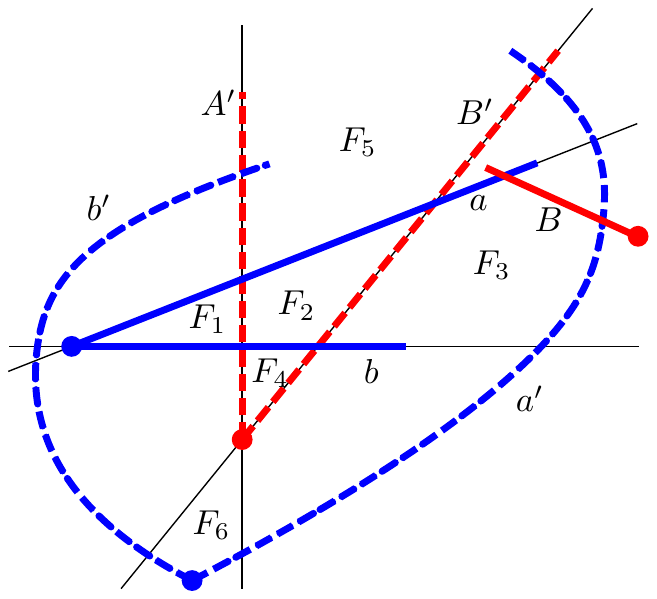}
    \caption{If $a'$ has an endpoint in $F_5$, then $B$ cannot intersect $b'$.}\label{fig:3_6b}
  \end{subfigure}
	\caption{Case~\ref{case-3-6}: $I(A,B) \in F_3$ and $I(a',b')
          \in F_6$.
        }
	\label{fig:3_6}
\end{figure} 

We have finished the case that
$a,b$ and $a',b'$ are hooking.
Suppose now that
$a,b$ and $a',b'$ are hooked, with respect to some pairs $A,B$ 
and $A',B'$.
Then $A,B$ is hooking with respect to $a,b$
and $A',B'$ is hooking with respect to $a',b'$.
Recall that $A$, $B$, $A'$ and $B'$ belong
to four different connected components. Hence, this case can be
handled as above,
after exchanging the capital letters with the small letters
(i.e., exchanging $P$ and~$Q$).
\end{proof}

\section{A Weaker Bound}
\label{sec:weaker}

The principal structure lemma is already powerful enough to get an
improvement over the previous best bound:
\begin{lemma}\label{cor:nCC}
$G\NI$ has at most $(n+5)/2$ connected components. 
\end{lemma}

\begin{proof}
Partition the sides $q_0,q_1,\ldots,q_{n-1}$
  of $Q$ into $(n-1)/2$ disjoint pairs
$q_{2i},q_{2i+1}$, discarding
the last side $q_{n-1}$.
Let $H_+$ denote the subset of these pairs that are hooked.
Suppose first that this set contains some pair
$q_{2i_0},q_{2i_0+1}$ of sides that are in two different connected
components.
Combining
 $q_{2i_0},q_{2i_0+1}$ with
any of the remaining pairs
$q_{2i},q_{2i+1}$ of $H_+$,
Lemma~\ref{lem:2} tells us that
the sides $q_{2i}$ and $q_{2i+1}$ 
must either belong to the same connected component,
or one of
them must belong to
$\CC(q_{2i_0})$ or $\CC(q_{2i_0+1})$.
In other words, each
 remaining pair contributes at most one ``new'' connected component, and
it follows that the sides in
$H_+$
belong to at most $|H_+|+1$ connected components.
 This conclusion holds
also in the case that $H_+$ contains no
pair
$q_{2i_0},q_{2i_0+1}$ of sides that are in different connected
components.

The same argument works for the complementary subset $H_-$ of
pairs that are not hooked, but hooking.
Along with $\CC(q_{n-1})$ there
are at most
$(|H_+|+1)+(|H_-|+1)+1=(n-1)/2+3 = (n+5)/2$ components.
\end{proof}

Together with Observation~\ref{obs:CCs}, this already
improves the previous bound
$mn-(m + \lceil \frac{n}{6} \rceil)$ 
for a large range of parameters, namely when $m\ge n\ge 11$:
\begin{proposition}\label{weak}
Let $P$ and $Q$ be simple polygons with $m$ and $n$ sides, respectively,
such that $m$ and $n$ are odd and $m \ge n \ge 3$.
Then there are at most $mn-(m+ \frac{n-5}2)$ intersection points between
$P$ and $Q$.
\qed
\end{proposition}

\section{Ramsey-Theoretic Tools}
\label{sec:Ramsey}

We recall some classic results.

A tournament is a directed graph that contains
between every pair of nodes $x,y$ either the arc $(x,y)$ or the arc
$(y,x)$ but not both.
A tournament is \emph{transitive} if for every three nodes $x,y,z$
the existence of the arcs $(x,y)$ and $(y,z)$ implies the existence of
the arc $(x,z)$. Equivalently, the nodes can be ordered on a line
such that all arcs are in the same direction.
The following is easy to prove by induction.
\begin{lemma}
[Erd\H os and Moser~\cite{Ramsey-tournaments}]
  \label{lem:sub-tournament}
Every tournament on a node set $V$ contains a transitive sub-tournament on $1+\lfloor\log_2|V|\rfloor$ nodes.
\end{lemma}
\begin{proof}
  Choose $v\in V$ arbitrarily, and let $N\subseteq V-\{v\}$ with
  $|N|\ge (|V|-1)/2$ be the set
of in-neighbors of $v$ or the set of out-neighbors of $v$, whichever
is larger. Then $v$ together with a transitive sub-tournament of $N$
gives a transitive sub-tournament of size one larger.
\end{proof}

A set of points $p_1,p_2,\ldots,p_r$ in the plane sorted by
$x$-coordinates
(and with distinct $x$-coordinates)
forms an \emph{$r$-cup} (resp., \emph{$r$-cap})  
if $p_i$ is below (resp., above) the line through $p_{i-1}$ and $p_{i+1}$ for every $1 < i < r$.

\begin{theorem}[Erd\H{o}s--Szekeres Theorem for caps and cups in point sets
  \cite{ES35}]\label{thm:ES-points}
   For any two integers $r\ge2$ and $s\ge2$,
   the value $ES(r,s):= \binom{r+s-4}{r-2}$
   fulfills the following
  statement\textup:

  Suppose that
  $P$ is a set of $ES(r,s)+1$ points in the plane with distinct
  $x$-coordinates such that
  no three points of $P$
  lie on a line.  Then $P$ contains an $r$-cup or an $s$-cap.

Moreover,
$ES(r,s)$
is the smallest value that fulfills the statement.
\qed
\end{theorem}

A similar statement holds for lines by the standard point-line duality.
A set of lines $\ell_1,\ell_2,\ldots,\ell_r$ sorted by slope
forms an \emph{$r$-cup}
(resp., \emph{$r$-cap})
if
$\ell_{i-1}$ and $\ell_{i+1}$ intersect
below (resp., above) $\ell_i$
for every $1 < i < r$.

\begin{theorem}[Erd\H{o}s--Szekeres Theorem for lines]\label{thm:ES-lines}
For the numbers $ES(r,s)$ from Theorem~\ref{thm:ES-points}, the
following statement holds
for any two integers $r\ge2$ and $s\ge2$\textup:

	Suppose that $L$ is a set of $ES(r,s)+1$ non-vertical lines in
        the plane no two of which are parallel and
        no three of which intersect at a common point.
	Then $L$ contains an $r$-cup or an $s$-cap.
        \qed
\end{theorem}

\begin{theorem}[Erd\H{o}s--Szekeres Theorem for monotone subsequences
\cite{ES35}]
  \label {thm:monotone}
For any integer $r\ge0$, a sequence of $r^2+1$ distinct numbers contains
either an increasing subsequence of length $r+1$
or a decreasing subsequence of length $r+1$.
        \qed
\end{theorem}

\section{Proof of Theorem~\ref{thm:main}}
\label{subsec:main-proof}
\subsection{Imposing More Structure on the Examples}
Going back to the proof of Theorem~\ref{thm:main}, recall that 
in light of Observation~\ref{obs:CCs} it is enough to prove that
$G\NI$, the disjointness graph of $P$ and $Q$, has at most constantly many connected components.

We will use the following constants:
$C_6 := 6$;
$C_5 := (C_6)^2+1=37$;
$C_4  := ES(C_5,C_5)+1
= \binom{70}{35}+1=
112{,}186{,}277{,}816{,}662{,}845{,}433 
<
2^{70}
$;
$C_3 := 2^{C_4-1}$;
$C_2 := C_3+5$;
$C_1 := 8C_2$;
$C := C_1-1<2^{2^{70}}$.








We claim that $G\NI$ has at most $C$ connected components.
Suppose that $G\NI$ has at least $C_1=C+1$ connected components,
numbered as $1,2,\ldots,C_1$.
For each connected component~$j$,
we find
two
consecutive sides $q_{i_j},q_{i_j + 1}$ of $Q$ such that
$\CC(q_{i_j})=j$ and $\CC(q_{i_j + 1})\ne j$.
We call $q_{i_j}$ the \emph{primary} side and $q_{i_j + 1}$ the \emph{companion}
side of the pair.
 We take these $C_1$ consecutive pairs in their cyclic order along $Q$
 and remove every second pair.
 This ensures that the remaining $C_1/2$ pairs are disjoint, in the
 sense
 that no side of $Q$ belongs to two different pairs.

We apply Lemma~\ref{lem:1} to each of the remaining $C_1/2$ pairs
$q_{i_j},q_{i_j+ 1}$
and find an associated pair $p_{k_j},p_{k_j \pm 1}$ such that $(q_{i_j},p_{k_j}), (q_{i_j+1},p_{k_j\pm 1}) \in E\NI$.
Therefore, $\CC(q_{i_j})=\CC(p_{k_j})$ and
$\CC(q_{i_j + 1})=\CC(p_{k_j\pm 1})\ne\CC(q_{i_j})$.
Again, we call $p_{k_j}$ the {primary} side and $p_{k_j \pm 1}$ the
{companion} side.
As before, we delete half of the pairs $p_{k_j},p_{k_j \pm 1}$
in cyclic order along~$P$, along with their associated pairs from $Q$, and thus we ensure that
the remaining $C_1/4$ pairs are disjoint also on $P$.

At least $C_1/8$ of the remaining pairs $q_{i_j},q_{i_j + 1}$ are
hooking or at least $C_1/8$ of them are hooked. We may assume that at
least $C_2 = C_1/8$ of the pairs $q_{i_j},q_{i_j + 1}$ are hooking with
respect to their associated pair, $p_{k_j},p_{k_j \pm 1}$, for
otherwise, $p_{k_j},p_{k_j\pm1}$ is hooking with respect to
$q_{i_j},q_{i_j + 1}$ and we may switch the roles of $P$ and~$Q$.
Let us denote by $Q_2$ the set of $C_2$ hooking consecutive pairs
$(q_{i_j},q_{i_j \pm 1})$ at which we have arrived.
(Because of the potential switch, we have to denote the companion side
by
$q_{i_j \pm 1}$ instead of $q_{i_j + 1}$ from now on.)

By construction,
all $C_2$ primary sides $q_{i_j}$ of these pairs belong to
  distinct components.
We now argue that all $C_2$ adjacent companion sides $q_{i_j\pm1}$
 with at most one exception
lie in the same
connected component, provided that $C_2\ge 4$.

We model the problem by a graph whose
nodes are the connected components of $G\NI$.
For each of the $C_2$ pairs
$q_{i_j},q_{i_j\pm1}$,
we insert an edge between
$\CC(q_{i_j})$ and $\CC(q_{i_j\pm1})$.
The result is a multigraph with $C_2$ edges and without loops.
Two disjoint edges would represent two consecutive pairs
of the form $(q_{i_j}
,q_{i_j\pm1})$ whose four sides are
in four distinct connected components, but this
is a contradiction to
Lemma~\ref{lem:2}.
Thus, the graph has no two disjoint edges, and
such graphs 
are easily classified:
they are the triangle (cycle 
on three vertices) and the star graphs $K_{1t}$, possibly
with multiple edges. Overall, the graph involves
 at least $C_2\ge4$ distinct connected components
$\CC(q_{i_j})$, and therefore the triangle graph is excluded.
Let $v$ be the central vertex of the star.
There can be at most one $j$ with
$\CC(q_{i_j})=v$, and we discard it.
All other sides $q_{i_j}$ have
$\CC(q_{i_j})\ne v$, and therefore
$\CC(q_{i_j\pm1})$ must be the other endpoint of the edge, that is,~$v$.

In summary, we have found $C_2-1$ adjacent pairs
$q_{i_j},q_{i_j \pm 1}$ with the following properties.
\begin{itemize}
\item The primary sides $q_{i_j}$ belong to
   $C_2-1$ distinct components.
\item All companion sides $q_{i_j\pm1}$ belong to the same component, distinct
  from the other   
   $C_2-1$ components.
 \item
All $2C_2-2$ sides of the pairs
   $q_{i_j},q_{i_j \pm 1}$ 
   are distinct. 
\item Each
  $q_{i_j},q_{i_j \pm 1}$ is hooking with respect to an associated pair
  $p_{k_j},p_{k_j\pm1}$.
 \item
All $2C_2-2$ sides of the pairs
$p_{k_j},p_{k_j\pm1}$  
   are distinct. 
\end{itemize}

Let us denote by $\chosensides$
the set of
$C_2-1$
sides $q_{i_j}$.

\begin{proposition}\label{prop:no-many-avoiding}
  There are no six distinct sides
$q_a,q_b, q_c, q_d, q_e, q_f$
among the
$C_2-1$
sides $q_{i_j}\in\chosensides$ 
such that
$q_a,q_b$ are avoiding or consecutive,
$q_c,q_d$ are avoiding or consecutive,
and $q_e,q_f$ are avoiding or consecutive.
\end{proposition}

\begin{proof}
  Suppose for contradiction that there are six such
  sides. 
  It follows from Lemma~\ref{lem:1} that there are two consecutive
  sides $p_{a'}$ and $p_{b'}$ of $P$ such that $\CC(p_{a'})= \CC(q_a)$
  and $\CC(p_{b'})= \CC(q_b)$.

  Similarly, we find
  a pair of consecutive sides $p_{c'}$ and $p_{d'}$ of $P$
  such that
  $
  \CC(p_{c'})= \CC(q_c)$ and
  $
  \CC(p_{d'})= \CC(q_d)$, and the same story for $e$
  and~$f$.
  By the pigeonhole principle, two of the three consecutive pairs
$(p_{a'},p_{b'})$,
$(p_{c'},p_{d'})$,
$(p_{e'},p_{f'})$
  are
hooking or two of them are hooked. This contradicts Lemma~\ref{lem:2}.
\end{proof}

Define a complete graph whose nodes are
the $C_2-1$ sides
$q_{i_j}\in\chosensides$,
and color an edge $(q_{i_j},q_{i_k})$ red if
$q_{i_j}$ and $q_{i_k}$ are avoiding or consecutive and blue otherwise.
Proposition~\ref{prop:no-many-avoiding}
says that this graph contains no red matching of size three.
This means that we can get rid of all red edges by removing at most 4 nodes.
To see this, pick any red edge and remove its two nodes from the graph.
If any red edge remains, remove its two nodes.
Then all red edges are gone, because otherwise we would find
a matching with three red edges.

We conclude that there is a blue clique of size
 $C_3=C_2-5$, i.e.,
 there is a set
 $Q_3\subset\chosensides$ of
 $C_3$ polygon sides among
the $C_2-1$ sides
 $q_{i_j}\in\chosensides$
 that are pairwise non-avoiding
and disjoint,
i.e., they do not
share a common endpoint.

Our next goal is to find a subset of 7 segments in $Q_3$ that are
arranged as in Figure~\ref{fig:six-sides}.  To define this precisely, we say for
 two segments $q$ and $q'$ that
$q$ \emph{stabs} $q'$ if $I(q,q') \in q'$.
Among any two non-avoiding and non-consecutive sides $q$ and $q'$,
either
$q$ {stabs} $q'$ or
$q'$ {stabs}~$q$, but not both.
Define a tournament $T$ whose nodes are
the $C_3$ sides
$q_{i_j}\in Q_3$,
and
the arc between 
 each pair of nodes is oriented towards the stabbed side.
It follows from Lemma~\ref{lem:sub-tournament} that $T$ has a
transitive sub-tournament of size $1+\lfloor\log_2 C_3\rfloor=C_4$.

Furthermore, since
$C_4  = ES(C_5,C_5)+1$,
it follows from Theorem~\ref{thm:ES-lines}
that
there is a subset of
$C_5$ sides such that the lines through them form a $C_5$-cup or a $C_5$-cap.
By a vertical
reflection
if needed,
we may assume that they
 form a $C_5$-cup.





We now reorder
these $C_5$ sides $q_{i_j}$ of $Q$ in stabbing order, according to the
transitive sub-tournament mentioned above.
By the
Erd\H{o}s--Szekeres Theorem on monotone subsequences
(Theorem~\ref {thm:monotone}),
there is a subsequence of size
$C_6+1=\sqrt{C_5-1}+1=7$ such that their slopes
form a monotone sequence.
By a 
horizontal
reflection
if needed, 
we may assume that they have
decreasing slopes.

We rename these $7$ segments 
to
$a_0,a_1,\ldots,a_6$,
and we denote the line $\ell(a_i)$ by $\ell_i$,
see
Figure~\ref{fig:six-sides}.
We have achieved the following properties:
\begin{itemize}
\item The lines $\ell_0,\ldots,\ell_6$ form a 7-cup,
  with decreasing slopes in this order.
\item The segments $a_i$ are pairwise disjoint and non-avoiding.
\item 
  $a_{i}$ stabs $a_{j}$ for every $i<j$.
\end{itemize}

\begin{figure}
	\centering
	\includegraphics[page=3]{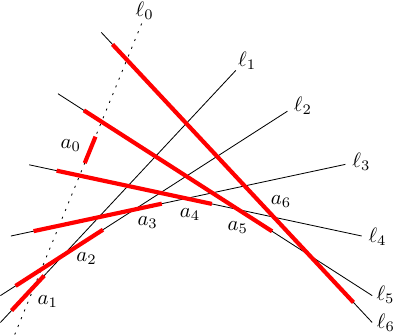}
	\caption{The seven sides $a_0,a_1,\ldots,a_6$. The lines
          $\ell_0,\ldots,\ell_6$ form a 7-cup.}
	\label{fig:six-sides}
\end{figure} 


These properties allow $a_0$
 to lie between any two consecutive intersections on
$\ell_0$. 
There is no such flexibility
for
the other sides:
Every side $a_j$ is stabbed by every preceding side $a_i$. 
For $1\le i<j$, $a_i$ cannot stab $a_j$ from the right,
because then $a_0$ would not be able to stab~$a_i$.
Hence,
the arrangement of the  sides
$a_1,\ldots,a_6$
must
be
exactly as shown in Figure~\ref{fig:six-sides}, in the sense
that the order of endpoints and intersection points along each line
$\ell_i$ is fixed.
We will ignore $a_0$ from now on.

%
          

\subsection{Finalizing the Analysis}
\label{sec:final-cases}

Recall that every $a_i$ is the primary side of two consecutive sides $a_i,b_i$
of $Q$
that are hooking 
with respect to an associated pair $A_i,B_i$ of consecutive sides of $P$.
The sides $a_i$ and $A_i$ are the primary sides and $b_i$~and $B_i$
are the companion sides.
All these $4\times 6
$ sides are distinct, and they intersect as follows:
$a_i$ intersects $B_i$ and is disjoint from $A_i$; $b_i$
intersects $A_i$ and is disjoint from $B_i$; and $I(A_i,B_i) \in
\Cone(a_i,b_i)$.

Figure~\ref{fig:subgraph-of-NI} summarizes the intersection pattern
among these sides.
A side $A_i$ must intersect every
side $a_j$ with $j \ne i$
and every side $b_j$
since
$\CC(A_i) = \CC(a_i) \ne \CC(a_j)$ and
$\CC(A_i) = \CC(a_i) \ne \CC(b_i)=\CC(b_j)$.
(Recall that all companion sides $b_i$ belong to the same component.)
Similarly, every side $B_i$ must intersect every side $a_j$.
We have no information about the intersection between $B_i$
and $b_j$, as these sides belong to the same connected
component.

\begin{figure}[htb]
  \centering
\includegraphics[scale=1.1]{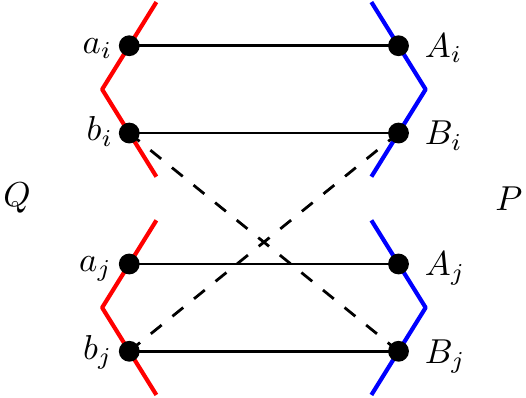}
\caption{The subgraph of $G\NI$ induced on two pairs of consecutive 
  sides $a_i,b_i$ and $a_j,b_j$ of $P$ and their associated partner pairs
  $A_i,B_i$ and $A_j,B_j$ of $Q$. Parts of $P$ and $Q$ are shown to
  indicate consecutive sides.
  The dashed edges may or may not be present.}
  \label{fig:subgraph-of-NI}
\end{figure}


We will now derive a contradiction through a series of case distinctions.

\subparagraph{Case 1:}
There are three segments $A_i$ with the property that
$A_i$ crosses $\ell_i$ to the left of~$a_i$.
Without loss of generality, assume that
 these segments are $A_1,A_2,A_3$, 
see Figure~\ref{fig:3_left}.
The segments $A_1,A_2,A_3$ must not cross because $P$ is a simple
polygon.
Therefore  $A_1$ intersects $a_2$ to the right of $I(a_1,a_2)$ because
otherwise $A_1$ would cross $A_2$ on the way between its intersections
with $\ell_2$ and with $a_1$.
$A_3$ must cross $\ell_3, a_2, a_1$ in this order, as shown.
But then $A_1$ and $A_3$ (and $a_2$) block $A_2$ from intersecting $a_3$.
\begin{figure}
	\centering
	\includegraphics[scale=1]{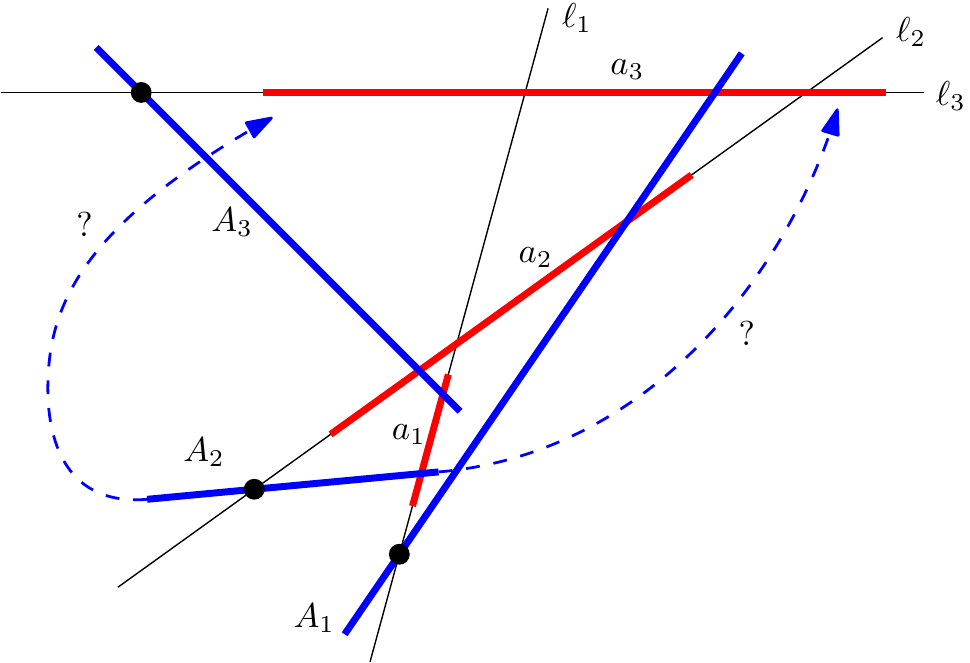}
	\caption{The assumed intersection points between $A_i$ and
          $\ell_i$ are marked.}
	\label{fig:3_left}
\end{figure} 

\goodbreak

\subparagraph{Case 2:}
There at most two segments $A_i$ with the property that
$A_i$ crosses $\ell_i$ to the left of~$a_i$.
In this case, we simply discard these segments.
We select four of the
remaining
segments
and renumber them from 1 to 4.

From now on, we can make the following assumption:
\begin{quote}
\textbf{General Assumption:}
For every $1 \le i \le 4$, the segment
$A_i$ does not cross $\ell_i$ at all, or it crosses
$\ell_i$ to the right of~$a_i$.
\end{quote}

This implies that $A_3$ must intersect the sides $a_2,a_1,a_4$ in this order,
and it is determined in which cell of the arrangement of the lines
$\ell_1,\ell_2,\ell_3,\ell_4$ the left endpoint of $A_3$ lies
(see Figures~\ref{fig:six-sides} and~\ref{fig:A3-below}). For the right endpoint, we have a choice
of two cells, depending on whether $A_3$ intersects $\ell_3$ or not.

We denote by $\mathrm{left}(s)$ and $\mathrm{right}(s)$ the left and
right endpoints of a segment $s$. 
We distinguish four cases, based on
whether
the common endpoint
of 
$A_3$
and
$B_3$ lies at
 $\mathrm{left}(A_3)$ or $\mathrm{right}(A_3)$,
and
whether
the common endpoint
of 
$a_3$
and
$b_3$ lies at
 $\mathrm{left}(a_3)$ or $\mathrm{right}(a_3)$.

\subparagraph{Case 2.1:}
$I(A_3,B_3) = \mathrm{left}(A_3)$ and
$I(a_3,b_3) = \mathrm{right}(a_3)$,
see Figure~\ref{fig:A3-below}.

As indicated in the figure, we leave it open whether and where $A_3$
intersects $\ell_3$.
We know that $b_3$ must lie below $\ell_3$ because
$I(A_3,B_3) \in\Cone(a_3,b_3)$.

\begin{figure}[htb]
	\centering
	\includegraphics[scale=1.1]{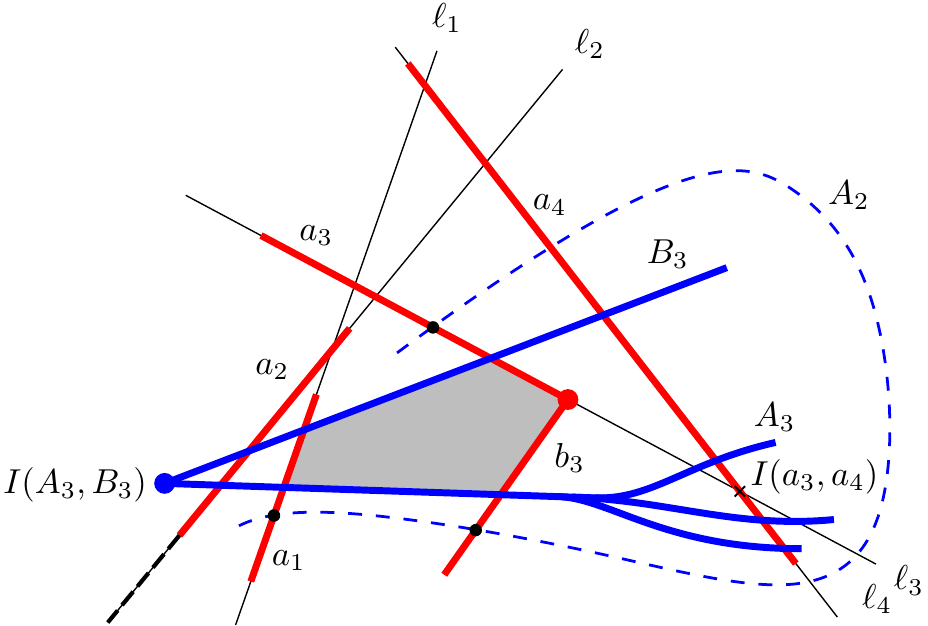}
	\caption{Case 2.1,
$I(A_3,B_3) = \mathrm{left}(A_3)$ and
$I(a_3,b_3) = \mathrm{right}(a_3)$.          
          A hypothetical segment $A_2$ is shown as a dashed curve.
The side $a_2$ and
        the part of $\ell_2$ to the left of $a_2$ is blocked for $A_2$.}
	\label{fig:A3-below}
\end{figure} 

We claim that $A_2$ cannot have the required intersections with $a_1$,
$a_3$, and $b_3$.
Let us first consider $a_1$: It is cut into three pieces by $A_3$ and
$B_3$.

If $A_2$ intersects the middle piece of $a_1$
in the wedge between $A_3$ and $B_3$, then $A_2$ intersects
exactly one of $a_3$ and $b_3$ inside the wedge, because these parts
together with $a_1$ are three sides of a convex pentagon.
If $A_2$ intersects $a_3$, then it has crossed $\ell_3$ and it cannot
cross $b_3$ thereafter.
If $A_2$ intersects $b_3$, it must cross $\ell_4$ before leaving the
wedge,
and then it cannot cross $a_3$ thereafter.

Suppose now that $A_2$ crosses the bottom piece of $a_1$. Then it cannot go
around $A_3,B_3$ to the right in order to reach $a_3$ because it would have
to intersect $\ell_4$ twice. $A_2$ also cannot pass to the left of
$A_3,B_3$ because
it cannot cross $\ell_2$ through $a_2$ or,
by the general assumption,  to the left of~$a_2$.

Suppose finally that $A_2$ crosses the top piece of $a_1$. Then it
would have to
cross $\ell_3$ twice before reaching $b_3$.

\goodbreak

\subparagraph{Case 2.2:}
$I(A_3,B_3) = \mathrm{left}(A_3)$ and $I(a_3,b_3) = \mathrm{left}(a_3)$.

If $\ell(A_3)$ does not intersect $a_3$,
we derive a contradiction as follows,
see Figure~\ref{fig:left-left-easy}.
We know that
the sides $a_2, a_3, a_4$ must be arranged as shown. 
The segment $A_3$ crosses $a_2$ 
but not~$a_3$.
Now, the parts of $a_3$ and $A_3$ to the left of $\ell_2$
form two opposite sides of a quadrilateral, as shown in the figure.
If this quadrilateral were not convex,
then either
$\ell(A_3)$ would intersect $a_3$, which we have excluded by assumption,
or $\ell_3$ would intersect $A_3$ left of $a_3$,
contradicting 
the General Assumption.
Thus, the sides
 $a_3$ and $A_3$ violate
the Axis Property (Observation~\ref{lem:axis}),
which requires
 $a_3$ and $A_3$ to lie on different sides of the line
through
$I(A_3,B_3)
$ and
$I(a_3,b_3)
$.

Looking back at this proof, we have seen that the configuration of the
segments $a_1,a_2, a_3, a_4$ according to Figure~\ref{fig:six-sides} in
connection with the particular case assumptions make the situation
sufficiently constrained that the case can be dismissed by looking at
the drawing.  The treatment of the other cases will be proofs by
picture in a similar way, but we will not always spell out the
arguments in such detail.

\begin{figure}[htb]
  \begin{minipage}[t]{0.47\linewidth}
    \centering
    \includegraphics[page=2]{right-right}
    \caption{Case 2.2. 
      $I(A_3,B_3) = \mathrm{left}(A_3)$, 
      $I(a_3,b_3) = \mathrm{left}(a_3)$,
      $\ell(A_3)$ does not intersect~$a_3$.}
    \label{fig:left-left-easy}
  \end{minipage}
\qquad
  \begin{minipage}[t]{0.49\linewidth}
    \centering
    \includegraphics{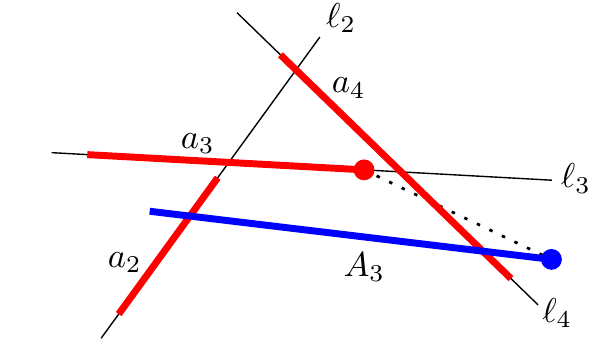}
    \caption{Case 2.3. $I(A_3,B_3) = \mathrm{right}(A_3)$,
      and
      $I(a_3,b_3) = \mathrm{right}(a_3)$,
      $A_3$ lies below $\ell_3$.}
    \label{fig:right-right}
  \end{minipage}
\end{figure}

\begin{figure}[htb]
	\centering
	\includegraphics[scale=1.05]{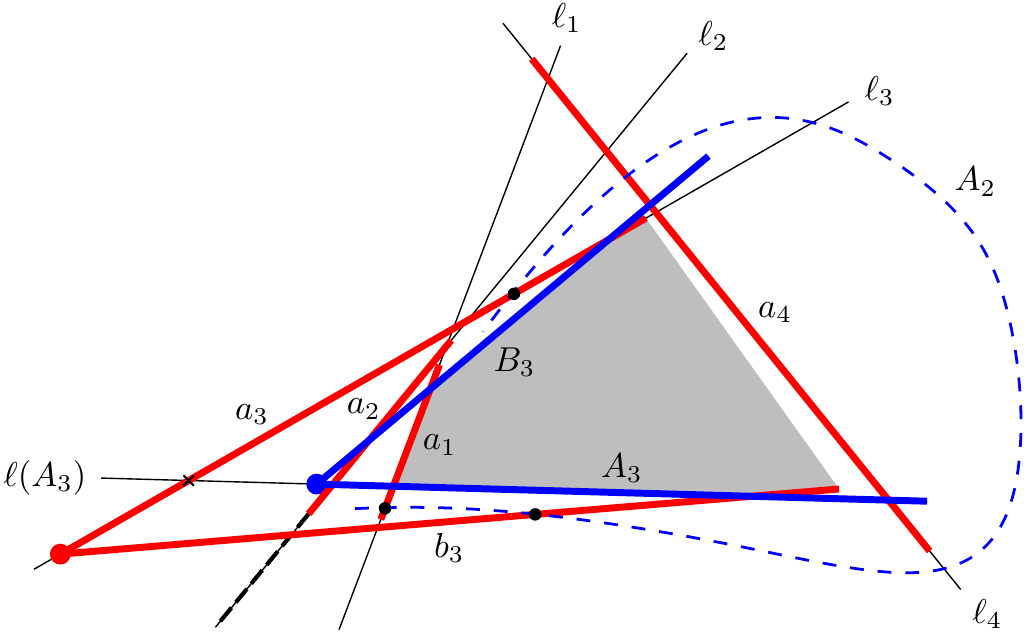}
	\caption{Case 2.2, $I(A_3,B_3) = \mathrm{left}(A_3)$,
          $I(a_3,b_3) = \mathrm{left}(a_3)$,
and $\ell(A_3)$ intersects $A_3$.
        A~hypothetical segment $A_2$ is shown as a dashed curve.}
	\label{fig:A3-left}
\end{figure} 

If $\ell(A_3)$ intersects $a_3$,
the situation must be as shown
in Figure~\ref{fig:A3-left}:
the pair $A_3,B_3$ is hooked
by 
$a_3$ and $b_3$.
The analysis of Case 2.1 
(Figure~\ref{fig:A3-below}) applies verbatim, except that the word
``pentagon'' must be replaced by ``hexagon''.

\subparagraph{Case 2.3:}
$I(A_3,B_3) = \mathrm{right}(A_3)$,
and
$I(a_3,b_3) = \mathrm{right}(a_3)$.

If  $A_3$ lies entirely below $\ell_3$, then
$A_3$ together with $a_3$ violates
the Axis Property (Observation~\ref{lem:axis}),
see Figure~\ref{fig:right-right}.

\begin{figure}[htb]
	\centering
	\includegraphics{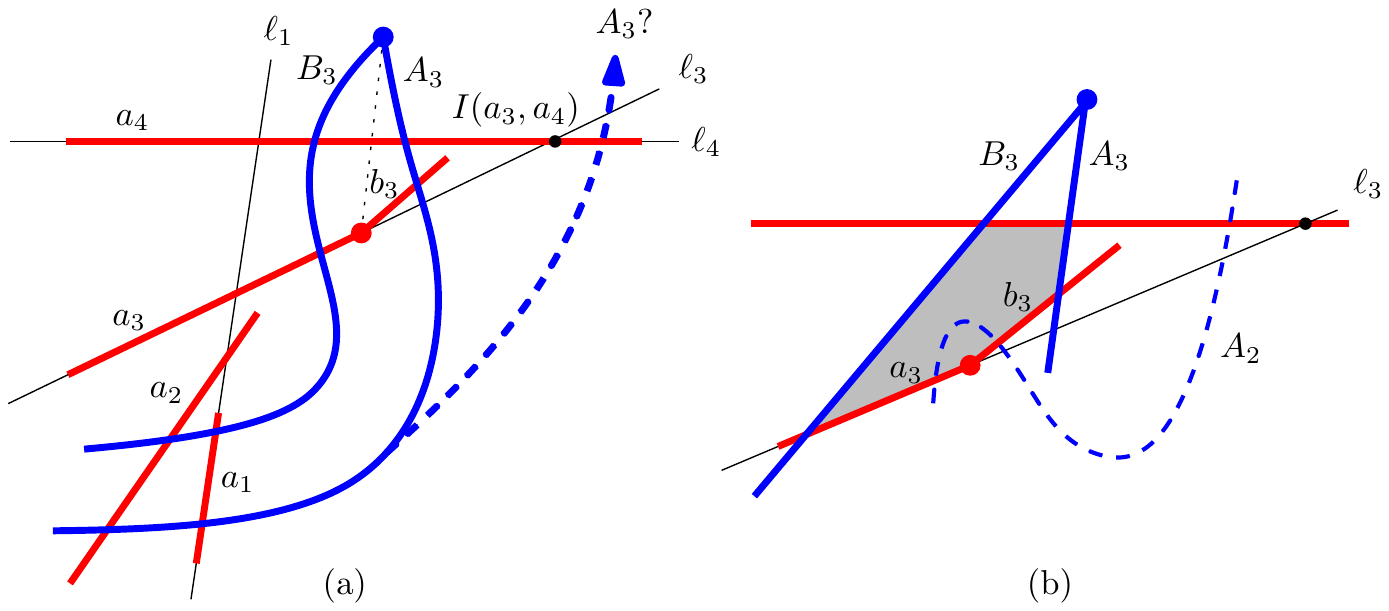}
	\caption{Case 2.3. $A_3$ intersects $\ell_3$.}
	\label{fig:cap-case-2_1_1}
\end{figure}

Let us therefore assume
that $A_3$ intersects $\ell_3$ (to the right of $a_3$),
and thus $\mathrm{right}(A_3)=I(A_3,B_3)$ lies above $\ell_3$,
see Figure~\ref{fig:cap-case-2_1_1}a.
Then $b_3$ must also lie above $\ell_3$, because $a_3,b_3$ is supposed
to be hooking, that is,
$I(A_3,B_3)\in \Cone(a_3,b_3)$.

It follows that $A_3$ cannot intersect $\ell_3$ to the right of
$I(a_3,a_4)$ (the option shown as a dashed curve),
because otherwise it would miss $b_3$: $b_3$ is blocked by $a_4$.

Therefore, the situation looks as shown in
Figure~\ref{fig:cap-case-2_1_1}a.
Figure~\ref{fig:cap-case-2_1_1}b shows the position of the relevant pieces.
The segments $a_4,B_3,a_3,b_3,A_3$ enclose a convex pentagon.
Now, the segment $A_2$ should intersect $a_3$, $b_3$, and $a_4$
without crossing $A_3$ and $B_3$, like the dashed curve
in the figure.
This is impossible. 

%



%

\subparagraph{Case 2.4:}
$I(A_3,B_3) = \mathrm{right}(A_3)$ and $I(a_3,b_3) =
\mathrm{left}(a_3)$.


If $A_3$ intersects $\ell_3$ (to the right of $a_3$), then
$A_3$ together with $a_3$ violates
the Axis Property (Observation~\ref{lem:axis}),
see Figure~\ref{fig:2_4-easy}.
We thus 
assume that $A_3$ lies entirely below~$\ell_3$.

\begin{figure}[htb]
	\centering
	\includegraphics{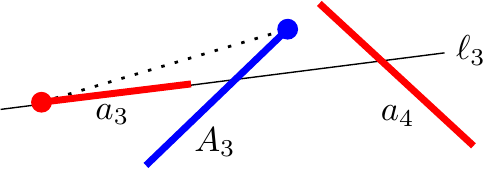}
	\caption{Case 2.4. $A_3$ intersects $\ell_3$.}
	\label{fig:2_4-easy}
\end{figure} 

\begin{figure}[htb]
	\centering
	\includegraphics[scale=0.98]{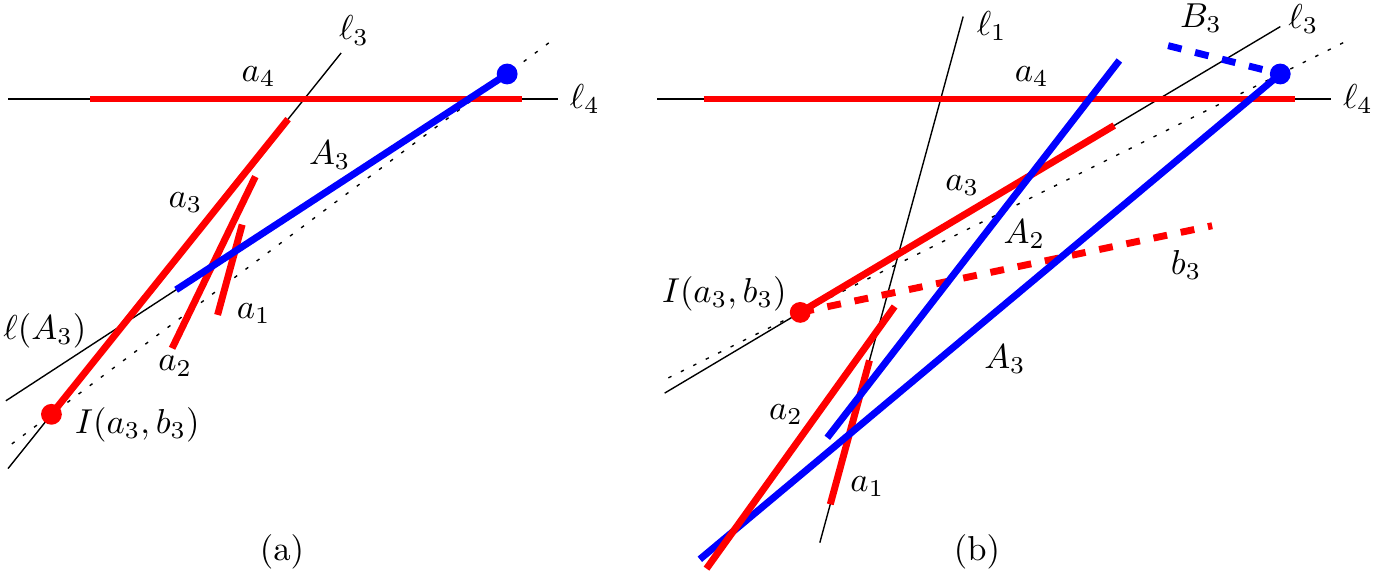}
	\caption{Case 2.4.  $A_3$ lies below $\ell_3$.}
	\label{fig:cap-case-2_4}
\end{figure} 

If  $\ell(A_3)$ passes above $I(a_3,b_3)=\mathrm{left}(a_3)$,
the sides $a_3$ and $A_3$ violate the Axis Property see
Figure~\ref{fig:cap-case-2_4}a.
On the other hand,
if  $\ell(A_3)$ passes below $I(a_3,b_3)=\mathrm{left}(a_3)$,
as shown in 
Figure~\ref{fig:cap-case-2_4}b,
then
 $b_3$
 must cross $\ell_1$ to the right of~$a_1$ in order to reach $A_2$.
Again by the Axis Property, $B_3$ must remain 
above the dotted axis line through
$I(A_3,B_3)=\mathrm{right}(A_3)$ and
$I(a_3,b_3)=\mathrm{left}(a_3)$.
On $\ell_1$, $b_3$ separates $a_1$ from the axis line, and hence
$a_1$ lies below the axis line.
Therefore $B_3$ and $a_1$ cannot
intersect.
%
%
%
%
%

This concludes the proof of Theorem~\ref{thm:main}. \qed
\subparagraph{Acknowledgement.} We thank the reviewers for helpful suggestions.

\bibliographystyle{abbrv}

\end{document}